\documentclass[a4paper,12pt,reqno]{amsart}

\usepackage[T1]{fontenc}
\usepackage{lmodern}
\usepackage{microtype}

\usepackage[left=3.2cm,right=3.2cm,top=2.8cm,bottom=2.8cm]{geometry}
\usepackage{amsmath}
\usepackage{amsfonts}

\usepackage{amssymb}
\usepackage{latexsym}
\usepackage{cite}
\usepackage{mathrsfs}
\usepackage{color}
\usepackage{enumitem}
\usepackage{graphicx}

\usepackage{cite}
\usepackage{marginnote}

\usepackage[colorlinks=true,urlcolor=blue,
citecolor=red,linkcolor=blue,linktocpage,pdfpagelabels,
bookmarksnumbered,bookmarksopen]{hyperref}
\usepackage[english]{babel}

\usepackage[hyperpageref]{backref}

\usepackage{ifdraft}
\ifdraft{
\usepackage[scrtime]{prelim2e}
\usepackage{refcheck}
\usepackage{cmdtrack}}

\usepackage{constants}

\numberwithin{equation}{section}

\newtheorem{theorem}{Theorem}[section]
\newtheorem{lemma}{Lemma}[section]
\newtheorem{corollary}{Corollary}[section]

\theoremstyle{definition}

\newtheorem{problem}{Problem}
\newcommand{\ud}{\mathrm{d}}
\newcommand{\RN}{\mathbb R^N}

\newcommand{\g}{\gamma}
\newcommand{\la}{\lambda}
\newcommand{\lb}{\Lambda}
\newcommand{\R}{\mathbb R}
\newcommand{\al}{\alpha}
\newcommand{\nor}[1]{\|#1\|}
\newcommand{\lan}{\langle}
\newcommand{\ran}{\rangle}
\newcommand{\e}{\varepsilon}
\newcommand{\vp}{\varphi}
\newcommand{\bc}{\begin{corollary}}
\newcommand{\ec}{\end{corollary}}
\newcommand{\bcs}{\begin{cases}}
\newcommand{\ecs}{\end{cases}}
\newcommand{\abs}[1]{\lvert#1\rvert}
\newcommand{\bigabs}[1]{\bigl\lvert#1\bigr\rvert}

\title[Groundstates for critical Choquard type equation]{Groundstates for Choquard type equations with Hardy--Littlewood--Sobolev lower critical exponent}

\author[D.\ Cassani]{Daniele Cassani}
\author[J.\ Van Schaftingen]{Jean Van Schaftingen}
\author[J.J.\ Zhang]{Jianjun Zhang}
\address[D.\ Cassani]{\newline\indent Dip. di Scienza e Alta Tecnologia
\newline\indent
Universit\`{a} degli Studi dell'Insubria
\newline\indent
and
\newline\indent
RISM - Riemann International School of Mathematics
\newline\indent
via G.B. Vico 46, 21100 Varese, Italy}
\email{\href{mailto:Daniele.Cassani@uninsubria.it}{Daniele.Cassani@uninsubria.it}}

\address[J. Van Schaftingen]{%
\newline\indent Institut de Recherche en Math\'{e}matique et Physique
\newline\indent Universit\'{e} Catholique de Louvain
\newline\indent Chemin du Cyclotron 2 bte L7.01.01
\newline\indent 1348 Louvain-la-Neuve, Belgium}
\email{\href{mailto:Jean.VanSchaftingen@UCLouvain.be}{Jean.VanSchaftingen@UCLouvain.be}}

\address[J. J.\ Zhang]{\newline\indent College of Mathematics and Statistics
\newline\indent
Chongqing Jiaotong University
\newline\indent
Chongqing 400074, PR China
\newline\indent and
\newline\indent Dip. di Scienza e Alta Tecnologia
\newline\indent
Universit\`{a} degli Studi dell'Insubria
\newline\indent
via G.B. Vico 46, 21100 Varese, Italy}
\email{\href{mailto:zhangjianjun09@tsinghua.org.cn}{zhangjianjun09@tsinghua.org.cn}}

\thanks{J.\thinspace Van Schaftingen was supported by the Projet de Recherche (Fonds de la Recherche Scientifique--FNRS) T.1110.14 ``Existence and asymptotic behavior of solutions to systems of semilinear elliptic partial differential equations''. J. J. Zhang was partially supported by the Science Foundation of Chongqing Jiaotong University(15JDKJC-B033).}


\begin{document}

\begin{abstract}
For the Choquard equation, which is a nonlocal nonlinear Schr\"o\-dinger type equation,
\begin{equation*}
-\Delta u+V_{\mu, \nu} u=(I_\alpha\ast |u|^{\frac{N+\alpha}{N}})\abs{u}^{\frac{\alpha}{N}-1}u,\quad \text{in $\mathbb{R}^N$,}
\end{equation*}
where $N\ge 3$, $V_{\mu, \nu} : \mathbb{R}^N \to \mathbb{R}$ is an external potential defined for $\mu,\nu\in\R$ and $x \in \mathbb{R}^N$ by $V_{\mu, \nu} (x)=1-\mu/(\nu^2 + \abs{x}^2)$ and $I_\alpha : \mathbb{R}^N \to 0$ is the Riesz potential for $\alpha\in (0,N)$, we exhibit two thresholds $\mu_{\nu},\mu^{\nu}>0$ such that the equation admits a positive ground state solution if and only if $\mu_{\nu}<\mu<\mu^{\nu}$ and no ground state solution exists for $\mu<\mu_{\nu}$. Moreover, if $\mu>\max\left\{\mu^{\nu},\frac{N^2(N-2)}{4(N+1)}\right\}$, then equation still admits a sign changing ground state solution provided $N\ge4$ or in dimension $N=3$ if in addition $\frac{3}{2}<\alpha<3$ and $\ker (-\Delta + V_{\mu, \nu}) = \{0\}\), namely in the non-resonant case. 
\end{abstract}

\maketitle

\section{Introduction and main results}

We consider the following class of equations 
\begin{equation}\label{peq}
-\Delta u+Vu=(I_{\alpha}\ast |u|^p)|u|^{p-2}u, \quad x\in \R^N
\end{equation}
where $V$ is an external potential, $N\geq 3$ and $I_\alpha$ is the Riesz potential given for each $x\in\RN\setminus\{0\}$ by 
$$
I_\alpha(x):=\frac{A_\alpha}{\abs{x}^{N-\alpha}},\quad \text{where}\quad  A_\alpha=\frac{\Gamma((N-\alpha)/2)}{\Gamma(\alpha/2)\pi^{N/2}2^\alpha}\quad \text{ and } \alpha\in (0,N),
$$
with $\Gamma$ being the Euler gamma function.

Equation \eqref{peq} with $p=2$, $N=3$ and $V$ constant, seems to appear in the literature in the 1950s with the work of S.~Pekar on quantum theory of the polaron \cite{Pekar}. Later rediscovered by Choquard in the context of Hartree-Fock theory of one component plasma and attracted attention of the mathematical community in late 1970s with the papers of E.~Lieb \cite{Lieb1} and P.L.~Lions \cite{Lions2, Lions3} which opened the way to an intensive study of \eqref{peq}, related problems, generalizations and extensions, see \cite{MV4} for a survey. Indeed, equations of the form \eqref{peq} show up as well in the context of self-gravitating matter \cite{Penrose1} and in the study of pseudo-relativistic boson stars \cite{boson}. The richness of plenty of applications have just in part contributed to the mathematical success and longevity of the interest in this kind of problems, as nonlocal Schr\"odinger type equations have been carrying over new mathematical challenges. Beyond physical motivations, ground state solutions to \eqref{peq} are of particular interest because of connections with stochastic analysis \cite{DonskerVaradhan2}.

The energy functional related to the Choquard equation \eqref{peq} is given for \(u : \RN \to \R\) by 
$$
E(u)=\frac{1}{2}\int_{\RN}\abs{\nabla u}^2+V\abs{u}^2-\frac{1}{2p}\int_{\RN} (I_\alpha\ast \abs{u}^{p})\abs{u}^{p}
$$
and in the case of constant potential $V\equiv 1$ it is well defined and $C^1(H^1(\R^N))$, by means of the Hardy--Littlewood--Sobolev inequality \cite{LL}, provided the following condition holds
\begin{equation}\label{subcritical_range}
\frac{N+\al}{N}\leq p\leq \frac{N+\al}{N-2}.
\end{equation}
Actually the range \eqref{subcritical_range} turns out to be sharp for the existence of variational solutions \cite{MV3}. Indeed, at the endpoints, sometimes called respectively lower and upper critical exponents for the Hardy--Littlewood--Sobolev inequality, a Pohozaev type identity prevent finite energy solutions to exist. As in the local Sobolev critical case, typical scaling invariance phenomena show up together with explicit one parameter family of extremal functions to the Hardy--Littlewood--Sobolev inequality, see Section \ref{Preliminary}.

The case in which a more general nonlinearity has upper critical growth has been recently considered in \cite{Cassani_Zhang} where Brezis--Nirenberg type perturbations allow to obtain ground state solutions as well as to study the singularly perturbed associated problem. In fact, in the upper critical case no reasonable perturbations of the constant potential have influence on the ground state energy level. The lack of compactness in the upper critical case occurs regardless of the properties of the external potential.  

In this paper, we consider the lower critical case, namely the following  problem
\begin{equation}\label{q1}
-\Delta u+V_{\mu, \nu} u=(I_\alpha\ast \abs{u}^{\frac{N+\alpha}{N}})\abs{u}^{\frac{\alpha}{N}-1}u,\quad \text{in \(\RN\)},
\end{equation}
where the external Schr\"odinger potential \(V_{\mu, \nu} : \R^N \to \R\) is a perturbation of the constant potential defined as 
$$
V_{\mu, \nu} =1-\frac{\mu}{\nu^2 + \abs{x}^2}, \quad \text{for } \mu,\nu\in\R \text{ and } x \in \R^N 
$$

This class of potentials has been considered in \cite{MV5} where the authors prove that \eqref{q1} admits a ground state solution provided $\frac{N^2(N-2)}{4(N+1)}<\mu\leq \nu^2$. The upper bound is equivalent to requiring $V_{\mu, \nu} \ge 0$ and this was not explicitly assumed in \cite{MV5}. Moreover, in \cite{MV5} is also proved that no nontrivial solutions do exist when $\mu<\frac{(N-2)^2}{4}$. 

A natural question is whether \eqref{q1} admits a ground state solution in the range 
$$\mu\in\left[\frac{(N-2)^2}{4},\frac{N^2(N-2)}{4(N+1)}\right].$$ Following Lions \cite{Lions1}, the strategy in \cite{MV5} to proving the existence of ground state solutions to \eqref{q1} is to establish the existence of minimizers to the constraint minimization problem
$$
  c_{\mu, \nu}:=\inf\left\{\int_{\RN}\abs{\nabla u}^2+V_{\mu, \nu}\abs{u}^2: G(u)=1,\,\ u\in H^1(\RN)\right\}
$$
where
$$
G(u):=\int_{\RN} (I_\alpha\ast \abs{u}^{\frac{N+\alpha}{N}})\abs{u}^{\frac{N+\alpha}{N}}
$$
and eventually removing the Lagrange multiplier by some appropriate scaling.

In Section~\ref{Preliminary} below, we show that $c_{\mu, \nu}\le0$ if $\mu\ge\mu^{\nu}$, where $\mu^{\nu}$ is the best constant of the embedding $H^1(\RN)\hookrightarrow L^2(\RN,(\nu^2 + \abs{x}^2)^{-1}\,\ud x)$. Precisely,
$$
  \mu^{\nu}:=\inf_{u\in H^1(\RN)\setminus\{0\}}\frac{\displaystyle \int_{\RN}\abs{\nabla u}^2+ \abs{u}^2}{\displaystyle \int_{\RN}\frac{\abs{u (x)}^2}{\nu^2 + \abs{x}^2}\,\ud x}.
$$
Clearly, $\mu^{\nu}\ge\nu^2$ and by Hardy's inequality, $\mu^{\nu}\ge(N-2)^2/4$. Actually, $\mu^{\nu}>\frac{(N-2)^2}{4} +\nu^2$, since the infimum $\mu^{\nu}$ is achieved by some function $u_0\in H^1(\RN)\setminus\{0\}$ (see Section~\ref{sectionProofTheorem1}) and then
\begin{equation}\label{lower_est_asym}
  \mu^{\nu}
  =\frac{\int_{\RN}\abs{\nabla u_0}^2+ \abs{u_0}^2}{\int_{\RN}\frac{1}{\nu^2 + \abs{x}^2}\abs{u_0}^2\,\ud x}
  >\frac{\int_{\RN}\abs{\nabla u_0}^2}{\int_{\RN}\frac{\abs{u_0 (x)}^2}{\abs{x}^2}\,\ud x}
  + \nu^2
  \ge\frac{(N-2)^2}{4} + \nu^2.
  \end{equation}

It can be proved as in \cite{MV5} that \(c_{\mu, \nu}\) is achieved and gives by rescaling a groundstate to \eqref{q1} provided  
$$\frac{N^2 (N - 2)}{4 (N + 2)} < \mu < \mu^\nu.$$ 

A natural question is whether \eqref{q1} admits a ground state solution when $\mu\ge\mu^{\nu}$. The main purpose of this paper is to make advances in the understanding of the picture presented above, by partially answering to the open questions raised so far and opening the way to new challenges.

Our main results are the following:

\begin{theorem}%
\label{Theorem 1}%
There exists a threshold $\mu_{\nu}\in[\frac{(N-2)^2}{4},\mu^{\nu})$
such that \eqref{q1} admits a positive ground state solution
if and only if $\mu_{\nu}<\mu<\mu^{\nu}$ and no ground state solution exists for $\frac{(N-2)^2}{4}<\mu<\mu_{\nu}$.
\end{theorem}

\begin{theorem}%
\label{Theorem 2}%
Assume that \(N \ge 4\) or \(N = 3\), $\frac{3}{2}<\alpha<3$ and $\ker (-\Delta + V_{\mu, \nu}) = \{0\}\).
If $\mu>\max\left\{\mu^{\nu},\frac{N^2(N-2)}{4(N+1)}\right\}$,
then \eqref{q1} admits a ground state solution (necessarily sign changing).
\end{theorem}

\begin{theorem}\label{Theorem 3} The best constant $\mu^{\nu}$ of the embedding $H^1(\RN)\hookrightarrow L^2(\RN,(\nu^2 + \abs{x}^2)^{-1}\,\ud x)$ enjoys the following 
$$
\lim_{N\to \infty}\frac{\mu^{\nu}}{\frac{N^2(N-2)}{4(N+1)}}=1.
$$
\end{theorem}

As a consequence of Theorem \ref{Theorem 3}, the four quantities $\mu_{\nu},\: \mu^{\nu},\: \frac{(N-2)^2}{4},\: \frac{N^2(N-2)}{4(N+1)}$ turn out to be asymptotically sharp as $N\to\infty$.

The statement of Theorem~\ref{Theorem 1} suggests the following open question.

\begin{problem}
Do groundstates solutions exist in the borderline cases $\mu=\mu_\nu$ and $\mu=\mu^\nu$?
\end{problem}

The statement of Theorem~\ref{Theorem 1} only makes sense when $\mu_{\nu}>\frac{(N-2)^2}{4}$.

\begin{problem}
In the case $\mu_{\nu}>\frac{(N-2)^2}{4}$, do there exist nontrivial solutions for \eqref{q1} for $\mu\in(\frac{(N-2)^2}{4},\mu_{\nu})$\,? 
\end{problem}

For $\nu=1$ and $N \in \{3,4,5\}$, one has 
\(\nu^2+\frac{(N-2)^2}{4}>\frac{N^2(N-2)}{4(N+1)}\) and so $\mu^{\nu}>\frac{N^2(N-2)}{4(N+1)}$. 

\begin{problem}
Does one have 
\[
\mu^{\nu}>\frac{N^2(N-2)}{4(N+1)}
\]
in general?
\end{problem}

A negative answer would suggest a further problem.

\begin{problem}
In the case $\mu^{\nu}<\frac{N^2(N-2)}{4(N+1)}$, do exist ground state solutions to \eqref{q1} when $\mu^{\nu}<\mu<\frac{N^2(N-2)}{4(N+1)}$\,?
\end{problem}

In Theorem~\ref{Theorem 2} the restriction $\frac{3}{2}<\alpha<3$ and $\ker (-\Delta + V_{\mu, \nu}) = \{0\}\) when \(N = 3\),
is only used to guarantee that the ground state level is strictly below the first level at which the Palais-Smale compactness condition fails (see Section~\ref{sectionProofTheorem2}). More precisely, if $\Lambda$ denotes the spectrum of the eigenvalue problem 
\begin{equation*}
-\Delta u+u=\frac{\la}{\nu^2 + \abs{x}^2}u,\quad u\in H^1(\RN),
\end{equation*}
in dimension $N=3$,  gives the existence of ground state solutions to \eqref{q1} in the non-resonant case, namely when $\mu\not\in\lb$. 

\begin{problem}
Does Theorem~\ref{Theorem 2} remain true in dimension \(N = 3\) in the resonant case \(\mu \in \Lambda\)?
\end{problem}

Finally, the necessity of the assumption on \(\alpha\) is not clear.

\begin{problem}
Does Theorem~\ref{Theorem 2} remain true in the resonant case in dimension $N=3$ for $\alpha\in (0,3/2)$\,?
\end{problem}


\section{Proof of Theorem \ref{Theorem 1}}\label{Preliminary}
\label{sectionProofTheorem1}

 Before proving Theorem~\ref{Theorem 1}, we introduce some preliminary results. First, the following Hardy--Littlewood--Sobolev inequality will be frequently used in the sequel.

\begin{lemma}%
[Hardy--Littlewood--Sobolev inequality {\cite[Theorem 4.3]{LL}}]
\label{hls}
Let $s, r>1$ and $0<\alpha<N$ with $1/s+1/r=1+\alpha/N$, $f\in L^s(\RN)$ and $g\in L^r(\RN)$, then there exists a positive constant $C(s, N, \alpha)$ (independent of $f, g$) such that
$$
\left|\int_{\RN}\int_{\RN}f(x)\abs{x - y}^{\alpha-N}g(y)\,\ud x\ud y\right|\le C(s, N, \alpha)\|f\|_s\|g\|_r.
$$
In particular, if $s=r=2N/(N+\alpha)$, the sharp constant is given by
$$
\mathcal{C}_\alpha:=\pi^{\frac{N-\alpha}{2}}\frac{\Gamma(\alpha/2)}{\Gamma((N+\alpha)/2)}\left[\frac{\Gamma(N/2)}{\Gamma(N)}\right]^{-\alpha/N}.
$$
\end{lemma}

The energy functional associated to the Choquard equation \eqref{q1} is given by 
$$
J_{\mu, \nu} (u)=\frac{1}{2}\int_{\RN}\abs{\nabla u}^2+V_{\mu, \nu}\abs{u}^2-\frac{N}{2(N+\alpha)}\int_{\RN} (I_\alpha\ast \abs{u}^{\frac{N+\alpha}{N}})\abs{u}^{\frac{N+\alpha}{N}},\,\, u\in H^1(\RN).
$$
Let
$$
c_\infty:=\inf\left\{\int_{\RN}\abs{u}^2: G(u)=1,\,\ u\in L^2(\RN)\right\},
$$
then it follows from \cite[Theorem 4.3]{LL} that $c_\infty>0$ and from \cite[Proposition 5]{MV5} that $c_{\mu, \nu}\le c_\infty$ for any $\mu,\nu > 0$. The following monotonicity property holds 
\begin{lemma}\label{monotonicity}
If \(\mu_1 \le \mu_2\) and $\nu_1 \ge \nu_2$, then $c_{\mu_1, \nu_1} \ge c_{\mu_2, \nu_2}$.
If moreover \(c_{\mu_1, \nu_1}\) is achieved, then equality holds if and only if \(\mu_1 = \mu_2\) and \(\nu_1 = \nu_2\).
\end{lemma}
\begin{proof}
For every \(u \in H^1 (\RN)\), we observe that
\[
 \int_{\RN}\abs{\nabla u}^2+V_{\mu_1, \nu_1}\abs{u}^2
 \ge \int_{\RN}\abs{\nabla u}^2+V_{\mu_2, \nu_2}\abs{u}^2;
\]
the conclusion follows by taking  the infimum over the functions \(u \in H^1 (\RN)\) that satisfy the constraint \(G (u) = 1\).

If we now assume that \(c_{\mu_1, \nu_1}\) is achieved for some function  \(u \in H^1 (\RN)\) such that \(G (u) = 1\), then
\[
 c_{\mu_2, \nu_2}
 \le \int_{\RN} \abs{\nabla u}^2+V_{\mu_2, \nu_2}\abs{u}^2
 = c_{\mu_1, \nu_1} - \int_{\RN} (V_{\mu_1, \nu_1} - V_{\mu_2, \nu_2}) \abs{u}^2,
\]
where the second integral will be positive if either \(\mu_1 < \mu_2\) or $\nu_1 > \nu_2$.
\end{proof}

The Nehari manifold $\mathcal{N}_{\mu, \nu}$ associated to \eqref{q1} is given by
$$
\mathcal{N}_{\mu, \nu}:=\left\{u\in H^1(\RN)\setminus\{0\}:\int_{\RN}\abs{\nabla u}^2+V_{\mu, \nu}\abs{u}^2=\int_{\RN} (I_\alpha\ast \abs{u}^{\frac{N+\alpha}{N}})\abs{u}^{\frac{N+\alpha}{N}}\right\}.
$$
It is easy to show that the set $\mathcal{N}_{\mu, \nu}$ is a $C^1$-manifold of codimension 1 for any $\mu,\nu > 0$. Moreover, $\mathcal{N}_{\mu, \nu}$ is regular in the sense that zero is isolated in $\mathcal{N}_{\mu, \nu}$. In fact, for any $u\in\mathcal{N}_{\mu, \nu}$,
\begin{equation*}
\begin{split}
\int_{\RN} (I_\alpha\ast \abs{u}^{\frac{N+\alpha}{N}})\abs{u}^{\frac{N+\alpha}{N}}=&\int_{\RN}\abs{\nabla u}^2+V_{\mu, \nu}\abs{u}^2\\
=&\frac{\mu}{\mu^{\nu}}\left[\int_{\RN}\abs{\nabla u}^2
+ \abs{u}^2-\mu^{\nu}\int_{\RN}\frac{\abs{u (x)}^2}{\nu^2 + \abs{x}^2}\,\ud x\right]\\
&\qquad +\Bigl(1- \frac{\mu}{\mu^{\nu}}\Bigr)\int_{\RN}\abs{\nabla u}^2+ \abs{u}^2\\
\ge&\Bigl(1- \frac{\mu}{\mu^{\nu}}\Bigr) \int_{\RN}\abs{\nabla u}^2+ \abs{u}^2.
\end{split}
\end{equation*}
By virtue of the Hardy--Littlewood--Sobolev inequality (Lemma~\ref{hls}), there exists $C>0$ depending on $\mu$ and $\nu$ such that $\nor{u}\ge C$ for any $u\in\mathcal{N}_{\mu, \nu}$. Let us denote the least energy by
$$
m_{\mu, \nu}:=\inf_{u\in\mathcal{N}_{\mu, \nu}}J_{\mu, \nu} (u).
$$

\begin{lemma}\label{lemma2} If $\mu<\mu^{\nu}$, then 
$$
m_{\mu, \nu}=\frac{\alpha}{2(N+\alpha)}c_{\mu, \nu}^{\frac{N+\alpha}{\alpha}}
$$
\end{lemma}
\begin{proof}
We observe that if $\mu<\mu^{\nu}$, then 
\[
\int_{\RN}\abs{\nabla u}^2+V_{\mu, \nu}\abs{u}^2>0
\]
for any $u\in H^1(\RN)\setminus\{0\}$.
 Let $u_t=tu$, where 
 \[t=\left(\int_{\RN}\abs{\nabla u}^2+V_{\mu, \nu}\abs{u}^2\right)^{-\frac{N}{2(N+\alpha)}},
 \]
 then $G(u_t)=1$ and by \eqref{yong},
\begin{equation*}
 \begin{split}
  m_{\mu, \nu} (u)&=\frac{\alpha}{2(N+\alpha)}\inf_{u\in\mathcal{N}_{\mu, \nu}}\left(\int_{\RN}|\nabla u_t|^2+V_{\mu, \nu}|u_t|^2\right)^{\frac{N+\alpha}{\alpha}}\\
&\ge\frac{\alpha}{2(N+\alpha)}c_{\mu, \nu}^{\frac{N+\alpha}{\alpha}}.
 \end{split}
\end{equation*}
On the other hand, for every $u\in H^1(\RN)$ with $G(u)=1$, let 
\[
  s=\left(\int_{\RN}\abs{\nabla u}^2+V_{\mu, \nu}\abs{u}^2\right)^{\frac{N}{2 \alpha}},
\]
 then $u_s=su\in\mathcal{N}$ and
\begin{equation*}
 \begin{split}
  \inf_{v\in\mathcal{N}}J(v)&=\frac{\alpha}{2(N+\alpha)}\inf_{v\in\mathcal{N}}\int_{\RN}\abs{\nabla v}^2+V_{\mu, \nu}\abs{v}^2\\
&\le\frac{\alpha}{2(N+\alpha)}\inf\left\{\int_{\RN}|\nabla u_s|^2+V_{\mu, \nu}|u_s|^2: G(u)=1,\,\ u\in H^1(\RN)\right\}\\
&=\frac{\alpha}{2(N+\alpha)}\inf\left\{\left(\int_{\RN}\abs{\nabla u}^2+V_{\mu, \nu}\abs{u}^2\right)^{\frac{N+\alpha}{\alpha}}: G(u)=1,\,\ u\in H^1(\RN)\right\}\\
&=\frac{\alpha}{2(N+\alpha)}c_{\mu, \nu}^{\frac{N+\alpha}{\alpha}}.\qedhere
 \end{split}
\end{equation*}
\end{proof}
\begin{lemma}\label{critical0}There exists \(u \in H^1 (\R^N)\) such that 
\[
\int_{\RN}\abs{\nabla u}^2+V_{\mu, \nu}\abs{u}^2
\]
and \(G (u) = 0\).
In particular, $c_{\mu^{\nu}, \nu}=0$.
\end{lemma}
\begin{proof}
Obviously, we have $c_{\mu^{\nu}, \nu}\ge 0$. By the definition of $\mu^{\nu}$, there exists a sequence $\{u_n\}_{n \in \mathbb{N}}$ in $H^1(\RN)$ such that as $n\to\infty$,
$$
\int_{\RN}\abs{\nabla u_n}^2+ \abs{u_n}^2\to \mu^{\nu},\quad \text{ and} \quad \int_{\RN}\frac{\abs{u_n (x)}^2}{\nu^2 + \abs{x}^2}\,\ud x=1.
$$
Up to a subsequence, there exists $u_0\in H^1(\RN)$ such that $u_n\rightharpoonup u_0$ weakly in $H^1(\RN)$ and almost everywhere in $\RN$, as $n\to\infty$. Noting that $1/(\nu^2 + \abs{x}^2)\to 0$, as $\abs{x}\to\infty$, we have by Rellich's compactnes theorem
\[
\int_{\RN}\frac{|u_0(x)|^2}{\nu^2 + \abs{x}^2}\,\ud x
= \lim_{n \to \infty} \int_{\RN}\frac{|u_n (x)|^2}{\nu^2 + \abs{x}^2}\,\ud x
=1
\]
and by the weak lower-semicontinuity of the norm, 
\[
\int_{\RN}\abs{\nabla u_0}^2+|u_0|^2\le \lim_{n \to \infty} \int_{\RN}\abs{\nabla u_n}^2+|u_n|^2 = \mu^{\nu}, 
\]
 which implies that $u_0$ is a minimizer of $\mu^{\nu}$. Therefore, $c_{\mu^{\nu}, \nu}=0$. Since \(u_0 \ne 0\), we reach the conclusion by a suitable scaling of \(u_0\).
\end{proof}

\begin{proof}[Proof of Theorem~\ref{Theorem 1}]
Since the Choquard equation \eqref{q1} does not admit any nontrivial solution if $\mu<\frac{(N-2)^2}{4}$, it follows that $c_{\mu, \nu}=c_\infty$ if $\mu<\frac{(N-2)^2}{4}$. Noting that $c_{\mu^{\nu}, \nu}=0$ and $c_{\mu, \nu}=c_\infty$ if $\mu<\frac{(N-2)^2}{4}$, by Lemma~\ref{monotonicity}, there exists $\mu_{\nu}\in[\frac{(N-2)^2}{4},\mu^{\nu})$ such that
$$
c_{\mu, \nu}=c_\infty \text{ if $0<\mu<\mu_{\nu}$}
\qquad \text{ and } \qquad c_{\mu, \nu}<c_\infty \text{ if $\mu_{\nu}<\mu<\mu^{\nu}.$ }
$$
Next we show that $c_{\mu, \nu}>0$ if $\mu_{\nu}<\mu<\mu^{\nu}$. Indeed, assume by contradiction that for some $\mu_{\nu}<\mu<\mu^{\nu}$, one has $c_{\mu, \nu}=0$.
Let $\{v_n\}_{n \in \mathbb{N}}$ be a minimizing sequence for $c_{\mu,\nu}$, then
\begin{equation*}
 \begin{split}
  0&=\lim_{n\to \infty}\int_{\RN}\abs{\nabla v_n}^2+V_{\mu, \nu}\abs{v_n}^2\\
&\ge\liminf_{n\to \infty}\int_{\RN}\abs{\nabla v_n}^2+V_{\mu^{\nu}, \nu}\abs{v_n}^2\ge c_{\mu^{\nu}, \nu}=0,
 \end{split}
\end{equation*}
which implies that 
\[
\lim_{n \to \infty} \int_{\RN}\frac{\abs{v_n (x)}^2}{\nu^2 + \abs{x}^2}\,\ud x = 0,
\]
and then $v_n\to 0$ strongly in $H^1(\RN)$, as $n\to\infty$, which contradicts the fact $G(v_n)=1$. Then for $\mu_{\nu}<\mu<\mu^{\nu}$, by virtue of \cite[Theorem 3]{MV5}, the level $c_{\mu, \nu}$ is achieved and $m_{\mu, \nu}$ as well.
\medbreak
Finally, to conclude the proof, it remains to proveing that for any $(N-2)^2/4<\mu<\mu_{\nu}$, $c_{\mu, \nu}$ cannot be achieved.  In fact, otherwise by the equality case in Lemma~\ref{monotonicity}, we would have
\(c_{\lambda, \nu} < c_\infty\) if \(\lambda > \mu\), contradicting the very definition of \(\mu_\nu\).
\end{proof}


\section{Proof of Theorem~\ref{Theorem 2}}
\label{sectionProofTheorem2}
For $\mu>\mu^{\nu}$, the set $\mathcal{N}_{\mu, \nu}$ is still a $C^1$-manifold of codimension 1. However, in contrast to the case $\mu<\mu^{\nu}$, the manifold $\mathcal{N}_{\mu, \nu}$ is no longer regular for $\mu>\mu^{\nu}$. Precisely, we next show that for $\mu>\mu^{\nu}$, $0$ is adherent to $\mathcal{N}_{\mu, \nu}$, which in turn implies $\inf_{u\in\mathcal{N}_{\mu, \nu}}J_{\mu, \nu} (u)\le0$. This is a main obstruction in proving the existence of a nontrivial critical point.  
\begin{lemma}\label{lemma1} 
If $\mu>\mu^{\nu}$, then 
$$
\inf_{u\in\mathcal{N}_{\mu, \nu}}J_{\mu, \nu} (u)=0.
$$
\end{lemma}
\begin{proof}
For $\mu>\mu^{\nu}$, we have
\begin{equation}
\label{yong} 
\begin{split}
\inf_{u\in\mathcal{N}_{\mu, \nu}}J_{\mu, \nu} (u)&=\frac{\alpha}{2(N+\alpha)}\inf_{u\in\mathcal{N}_{\mu, \nu}}\int_{\RN}\abs{\nabla u}^2+V_{\mu, \nu}\abs{u}^2\\
&=\frac{\alpha}{2(N+\alpha)}\inf_{u\in\mathcal{N}_{\mu, \nu}}\left[\frac{\int_{\RN}\abs{\nabla u}^2+V_{\mu, \nu}\abs{u}^2}{\left(\int_{\RN} (I_\alpha\ast \abs{u}^{\frac{N+\alpha}{N}})\abs{u}^{\frac{N+\alpha}{N}}\right)^{\frac{N}{N+\alpha}}}\right]^{\frac{N+\alpha}{\alpha}}\\
&=:\frac{\alpha}{2(N+\alpha)}d_\mu^{\frac{N+\alpha}{\alpha}}.
 \end{split}
\end{equation}
Next we show that actually $d_\mu=0$. Clearly,
\begin{multline*}
d_\mu=\inf\biggl\{\int_{\RN}\abs{\nabla u}^2+V_{\mu, \nu}\abs{u}^2: u\in H^1(\RN),
 \int_{\RN}\abs{\nabla u}^2+V_{\mu, \nu}\abs{u}^2>0\\ \text{ and } G(u)=1\biggr\}.
\end{multline*}
Let $u_0$ be the minimizer of $c_{\mu^{\nu}, \nu}$ obtained in Lemma~\ref{critical0}, then $\int_{\RN}\abs{\nabla u_0}^2+V_{\mu, \nu}|u_0|^2<0$ for any $\mu>\mu^{\nu}$. On the other hand, for any $\mu>\mu^{\nu}$, there exists $u_1\in H^1(\RN)$ such that
$$
\int_{\RN}|\nabla u_1|^2\ge(\mu+2)\int_{\RN}|u_1|^2>0,
$$
which implies $\int_{\RN}|\nabla u_1|^2+V_{\mu, \nu}|u_1|^2>0$. By scaling, there exist $\Tilde{u}_0,\Tilde{u}_1\in H^1(\RN)$ such that $G(\Tilde{u}_i)=1,\,\, i=0,1$ and
$$
\int_{\RN}|\nabla \Tilde{u}_0|^2+V_{\mu, \nu}|\Tilde{u}_0|^2<0,\,\,\int_{\RN}|\nabla \Tilde{u}_1|^2+V_{\mu, \nu}|\Tilde{u}_1|^2>0.
$$
Let $\g(t)=t\Tilde{u}_0+(1-t)\Tilde{u}_1$, $t\in[0,1]$, then $0<\inf_{t\in[0,1]}\|\g(t)\|\le\sup_{t\in[0,1]}\|\g(t)\|<+\infty$. Let $\Tilde{\g} (t)=\g(t)[G(\g(t))]^{-\frac{N}{2(N+\alpha)}}$, then $G(\Tilde{\g} (t))=1$ for $t\in[0,1]$ and
$$
\int_{\RN}|\nabla \Tilde{\g} (0)|^2+V_{\mu, \nu}|\Tilde{\g} (0)|^2>0,\,\,\, \int_{\RN}|\nabla \Tilde{\g} (1)|^2+V_{\mu, \nu}|\Tilde{\g} (1)|^2<0.
$$
By the continuity of $\Tilde{\g}$, there exists a sequence $\{\Tilde{\g} (t_n)\}_{n \in \mathbb{N}}$ such that $\int_{\RN}|\nabla \Tilde{\g} (t_n)|^2+V_{\mu, \nu}|\Tilde{\g} (t_n)|^2\to 0^+$, as $n\to\infty$, where $t_n\in[0,1]$. As a consequence, we get $d_{\mu}=0$ and there exists $\{\la_n\}$ with $\la_n\Tilde{\g} (t_n)\in\mathcal{N}_{\mu, \nu}$ and $\la_n\to 0$, as $n\to\infty$. 
\end{proof}
\bc
As a byproduct from the proof of Lemma~\ref{lemma1}, we have $c_{\mu,\nu}<0$ if $\mu>\mu^{\nu}$.
\ec

\subsection{Eigenvalues and eigenfunctions.} Consider the eigenvalue problem
\begin{equation*}
-\Delta u+u=\frac{\la}{\nu^2 + \abs{x}^2}u,\,\,u\in H^1(\RN).
\end{equation*}
Obverse that this problem is equivalent to the eigenvalue problem $T(u)=\frac{1}{\la}u$, where $T: H^1(\RN)\to H^1(\RN)$,
$$
T(u):=(-\Delta+I)^{-1}\circ\mathcal{K} (u),\,\, \mathcal{K} (u)=u/(\nu^2 + \abs{x}^2),\,\,u\in H^1(\RN).
$$
Since $1/(\nu^2 + \abs{x}^2)\to 0$, the linear operator  $\mathcal{K}: H^1(\RN)\to L^2(\RN)$ is compact. It is well known that $(-\Delta+I)^{-1}: L^2(\RN)\to H^1(\RN)$ is continuous. Then $T$ is compact and self-adjoint, which implies that there exist a sequence of eigenvalues $\{\la_n\}$ with finite multiplicity and such that
$$
0<\mu^{\nu}=\la_1<\la_2\le\cdots\le\la_n\le\cdots \to +\infty,\,\, n\to \infty
$$
and the associated normalized eigenfunctions $\{\vp_n\}$ satisfy for any $i,j\in\mathbb{N}$ and $i\not=j$,
$$
\int_{\RN}\nabla \vp_i\nabla \vp_j+\vp_i\vp_j=0,\,\,\int_{\RN}\frac{1}{\nu^2 + \abs{x}^2}|\vp_i|^2\,\ud x=1.
$$
Moreover, noting that for any $n$,
$$
\lim_{a\to 0^+}\sup_{x\in\RN}\int_{B_a(x)}\frac{1}{\abs{x - y}^{N-2}}\Bigl(1+\frac{\lambda_n}{(1+|y|^2)}\Bigr)\,\ud y=0,
$$
by virtue of \cite[Theorem C.2.5, C.3.4]{Simon}, there exist $C_n,\delta_n>0$ such that
$$
|\vp_n(x)| + |\nabla \vp_n(x)|\le C_n\exp{(-\delta_n\abs{x})},\,\, x\in\RN.
$$
From now on, we assume $\mu>\frac{N^2(N-2)}{4(N+1)}$ and 
$\mu\in[\la_n,\la_{n+1})$ if $N\ge4$ or $\mu\in(\la_n,\la_{n+1})$ if $N\ge3$. Let
$$
E^-={\rm span}\{\vp_1,\vp_2,\cdots,\vp_n\},\,\, E^+=\overline{{\rm span}\{\vp_{n+1},\vp_{n+2},\cdots\}},
$$
then $H^1(\RN)=E^-\oplus E^+$. 
\subsection{Energy estimates.} For any $\varepsilon > 0$, set
$$
u_{\varepsilon} (x)=\varepsilon^{\frac{N}{2}}U(\e x),
$$
where $U(x)=C\nu^{\frac{N}{2}} (\nu^2 + \abs{x}^2)^{-\frac{N}{2}}$ is a minimizer of $c_\infty$.
Following \cite{MV5}, we have
\begin{lemma}\label{lemma4}
Assume $\mu>\frac{N^2(N-2)}{4(N+1)}$. Then, for $\e>0$ small enough, we have 
$$
\int_{\RN}\abs{\nabla u_{\varepsilon}}^2 +V_{\mu, \nu}\abs{u_{\varepsilon}}^2>0,
$$
and
$$
\lim_{\varepsilon \to 0}\varepsilon^{-2}\int_{\RN}\Bigl[\abs{\nabla u_{\varepsilon} (x)}^2-\frac{\mu}{\nu^2 + \abs{x}^2}\abs{u_{\varepsilon} (x)}^2\Bigr]\,\ud x<0.
$$
\end{lemma}
\begin{proof} Observe that
$$
\int_{\RN}\abs{\nabla u_{\varepsilon}}^2=\varepsilon^2\int_{\RN}|\nabla U|^2<+\infty,\,\,\int_{\RN} (I_\alpha\ast \abs{u_{\varepsilon}}^{\frac{N+\alpha}{N}})\abs{u_{\varepsilon}}^{\frac{N+\alpha}{N}}=1
$$
and
$$
\int_{\RN}\abs{u_{\varepsilon}}^2=c_\infty,\,\,\,\int_{\RN}\abs{\nabla u_{\varepsilon}}^2 +V_{\mu, \nu}\abs{u_{\varepsilon}}^2=c_\infty+\varepsilon^{2}\mathcal{I}_\mu(\varepsilon),
$$
where
$$
\mathcal{I}_\mu(\varepsilon)=\varepsilon^{-2}\int_{\RN}\Bigl[\abs{\nabla u_{\varepsilon} (x)}^2-\frac{\mu \abs{u_{\varepsilon} (x)}^2}{\nu^2 + \abs{x}^2}\Bigr]\,\ud x.
$$
For some $c>0$, by Lebesgue's dominated convergence theorem, if $\mu>\frac{N^2(N-2)}{4(N+1)}$ we obtain
\begin{equation*}
\begin{split}
\mathcal{I}_\mu(\varepsilon)&=\nu^{-2}\int_{\RN}\Bigl[\frac{N^2(N-2)}{4(N+1)}-\frac{\mu\abs{x}^2}{\varepsilon^2+ \abs{x}^2}\Bigr]\frac{c}{\abs{x}^2(1 + \abs{x}^2)^N}\,\ud x\\
&\to \nu^{-2}\Bigl[\frac{N^2(N-2)}{4(N+1)}-\mu\Bigr]\int_{\RN}\frac{c}{\abs{x}^2(1 + \abs{x}^2)^N}\,\ud x<0,\,\,\text{as $\varepsilon \to 0$.}\qedhere
\end{split}
\end{equation*}
\end{proof}
Let us recall for the convenience of the reader the following two elementary lemmas:
\begin{lemma}{\cite{MMVS}}\label{norm}
Define
$$
\nor{u}_\ast:=\left[\int_{\RN} (I_\alpha\ast \abs{u}^{\frac{N+\alpha}{N}})\abs{u}^{\frac{N+\alpha}{N}}\right]^{\frac{N}{2(N+\alpha)}},\,\, u\in L^2(\RN),
$$
then $\|\cdot\|_\ast$ is a norm in $L^2(\RN)$.
\end{lemma}
\begin{proof}
We need to prove the triangle inequality.
By the Hardy--Littlewood--Sobolev inequality, for any $u\in L^2(\RN)$ one has 
$$
\int_{\RN} (I_\alpha\ast \abs{u}^{\frac{N+\alpha}{N}})\abs{u}^{\frac{N+\alpha}{N}}\le\mathcal{C}_\alpha\nor{u}_2^{\frac{2(N+\alpha)}{N}}<+\infty.
$$
We also observe that by the semigroup property of Riesz potentials, we have
\[
 \nor{u}_* = \biggl(\int_{\RN} \abs{I_{\alpha/2} \ast \abs{u}^\frac{N + \alpha}{N}}^2 \biggr)^\frac{N}{2(N + \alpha)}.
\]
Let \(u, v \in L^2 (\RN)\).
By the Minkowski integral inequality, we have for almost every \(x \in \RN\),
\[
 \bigl(I_{\alpha/2} \ast \abs{u + v}^\frac{N + \alpha}{N}\bigr)^\frac{N}{N + \alpha}
 \le
 \bigl(I_{\alpha/2} \ast \abs{u}^\frac{N + \alpha}{N}\bigr)^\frac{N}{N + \alpha}
 + \bigl(I_{\alpha/2} \ast \abs{v}^\frac{N + \alpha}{N}\bigr)^\frac{N}{N + \alpha},
\]
and thus by integrating and using again the Minkowski inequality we get
\[
\begin{split}
  \nor{u + v}_* &\le \biggl(\int_{\RN} \bigl(\bigl(I_{\alpha/2} \ast \abs{u}^\frac{N + \alpha}{N}\bigr)^\frac{N}{N + \alpha}
 + \bigl(I_{\alpha/2} \ast \abs{v}^\frac{N + \alpha}{N}\bigr)^\frac{N}{N + \alpha} \bigr)^\frac{2(N + \alpha)}{N}\bigr) \biggr)^\frac{N}{2(N + \alpha)}\\
 &\le \biggl(\int_{\RN}\abs{I_{\alpha/2}\ast \abs{u}^{\frac{N+\alpha}{N}})}^2\biggr)^{\frac{N}{2(N+\alpha)}}
 + \biggl(\int_{\RN}\abs{I_{\alpha/2}\ast \abs{v}^{\frac{N+\alpha}{N}}}^2 \biggr)^{\frac{N}{2(N+\alpha)}}\\
 &= \nor{u}_* + \nor{v}_*.\qedhere
\end{split}
\]
\end{proof}
\begin{lemma}
\label{pointInequality}
For $p\in(1,2)$ and \(a, b \in \R\), we have
\[
 |\abs{a}^p + \abs{b}^p -\abs{a + b}^p| \le 2^{2 - p} p\, \abs{a}\, \abs{b}^{p - 1}.
\]
\end{lemma}
\begin{proof}
We have
\[
 \abs{a}^p + \abs{b}^p - \abs{a + b}^p
 =p a \int_0^1 \abs{t a}^{p - 2} t a - \abs{t a + b}^{p - 2} (t a + b)\,\ud t.
\]
We observe that the integrand is maximal when \(ta = \frac{b}{2}\), and thus
\[
 \bigabs{\abs{t a}^{p - 2} t a - \abs{t a + b}^{p - 2} (t a + b)}
 \le 2^{2 - p} \abs{b}^{p - 1}.\qedhere
\]
\end{proof}

\begin{lemma}\label{estimate}
Let $\mu>\frac{N^2(N-2)}{4(N+1)}$. If $\mu\in[\la_n,\la_{n+1})$ when $N=4$ or $\mu\in(\la_n,\la_{n+1})$ when $N = 3$, we have for $\varepsilon > 0$ small enough,
$$
\sup_{w\in\hat{E} (u_{\varepsilon})}J_{\mu, \nu} (w)<\frac{\alpha}{2(N+\alpha)}c_\infty^{\frac{N+\alpha}{\alpha}},
$$
where $\hat{E} (u_{\varepsilon}):=\{w\in H^1(\RN): w=tu_{\varepsilon} +v,\,\,t\ge 0,\,\,v\in E^-\}$.
\end{lemma}

\begin{proof}
\resetconstant
By Lemma~\ref{pointInequality}, we have for every $x\in\RN$, $\varepsilon > 0$, $v\in E^-$ and $t > 0$

$$
|\abs{t u_{\varepsilon} + v}^{\frac{N+\alpha}{N}}
-\abs{t u_{\varepsilon}}^{\frac{N+\alpha}{N}}-\abs{v}^{\frac{N+\alpha}{N}}|
\le \Cl{cstxnyo} \abs{t u_{\varepsilon}} \abs{v}^{\frac{\alpha}{N}},
$$
where $C_1=2^{(N-\al)/N} (N+\al)/N$.

From which  we get 
\begin{multline}
\label{esti}
\int_{\RN} (I_\alpha\ast \abs{t u_{\varepsilon} + v}^{\frac{N+\alpha}{N}})\abs{t u_{\varepsilon} + v}^{\frac{N+\alpha}{N}}\\
=t^{\frac{2(N+\alpha)}{N}} +2t^{\frac{N+\alpha}{N}}\int_{\RN} (I_\alpha\ast \abs{u_{\varepsilon}}^{\frac{N+\alpha}{N}})\abs{v}^{\frac{N+\alpha}{N}} +\int_{\RN} (I_\alpha\ast \abs{v}^{\frac{N+\alpha}{N}})\abs{v}^{\frac{N+\alpha}{N}}\\
- 2 \Cr{cstxnyo} \int_{\RN} (I_\alpha \ast (\abs{t u_\varepsilon}^\frac{N + \alpha}{N} + \abs{v}^\frac{N + \alpha}{N})) \abs{t u_\varepsilon} \abs{v}^\frac{\alpha}{N},
\end{multline}
By the Hardy--Littlewood--Sobolev inequality, we have
\begin{multline*}
\int_{\RN} (I_\alpha \ast \abs{t u_\varepsilon}^\frac{N + \alpha}{N} + \abs{v}^\frac{N + \alpha}{N}) \abs{t u_\varepsilon} \abs{v}^\frac{\alpha}{N}\\
\le \C (\nor{t u_\varepsilon}_L^2 + \nor{v}_{L^2} )^\frac{N + \alpha}{N} \Bigl(\int_{\RN}
\abs{u_\varepsilon}^\frac{2N}{N + \alpha}\abs{v}^\frac{2 \alpha}{N + \alpha} \Bigr)^\frac{N + \alpha}{2N}
\end{multline*}
Recalling that $|u_{\varepsilon} (x)|\le \C \varepsilon^{N/2}$ for any $x\in\RN$ and that all the norms in $E^-$ are equivalent (since ${\rm dim} (E^-)<+\infty$) and $E^- \subset L^{2 \alpha/(N + \alpha)} (\RN)$ for any $r>0$, we have
\[
 \int_{\RN} (I_\alpha \ast \abs{t u_\varepsilon}^\frac{N + \alpha}{N} + \abs{v}^\frac{N + \alpha}{N}) \abs{t u_\varepsilon} \abs{v}^\frac{\alpha}{N}
 \le \C (t^\frac{N + \alpha}{N} \nor{v}^{\frac{\alpha}{N}} + \nor{v}^\frac{N+ 2\alpha}{N}) t \varepsilon^\frac{N}{2}.
\]

For any $t>0$,
\begin{multline*}
J_{\mu, \nu} (tu_{\varepsilon} +v)\le\frac{t^2}{2}\int_{\RN}\abs{\nabla u_{\varepsilon}}^2+V_{\mu, \nu}\abs{u_{\varepsilon}}^2+t\int_{\RN}\nabla u_{\varepsilon}\cdot \nabla v+V_{\mu, \nu} u_{\varepsilon} v\\
-\frac{N}{2(N+\alpha)}\int_{\RN} (I_\alpha\ast \abs{t u_{\varepsilon} + v}^{\frac{N+\alpha}{N}})\abs{t u_{\varepsilon} + v}^{\frac{N+\alpha}{N}}.
\end{multline*}
Notice that $\abs{\nabla u_{\varepsilon}} \le C \varepsilon^{\frac{N}{2} + 1}\) and $\abs{\nabla v}\in L^r(\RN)$ for any $r>0$. By Lemma~\ref{lemma4}
\begin{equation*}
\left|\int_{\RN}\nabla u_{\varepsilon}\cdot \nabla v+V_{\mu, \nu}u_{\varepsilon} v\right|
\le\|\nabla u_{\varepsilon}\|_{\infty}\|\nabla v\|_{L^1}+(1+\mu)\|u_{\varepsilon}\|_{\infty}\nor{v}_1
\le \Cl{cstNq} \varepsilon^{\frac{N}{2}} \nor{v}.
\end{equation*}
Then we have 
\begin{multline*}
J_{\mu, \nu} (tu_{\varepsilon} +v)
\le\frac{c_\infty}{2}t^2-\Cl{cstGap} t^2 \varepsilon^2+\Cr{cstNq} t \varepsilon^{\frac{N}{2}} \nor{v}\\
-\frac{N}{2(N+\alpha)}\int_{\RN} (I_\alpha\ast \abs{t u_{\varepsilon} + v}^{\frac{N+\alpha}{N}})\abs{t u_{\varepsilon} + v}^{\frac{N+\alpha}{N}}.
\end{multline*}
By Lemma~\ref{norm} and by \eqref{esti}, there exists $\Cl{cstEquiv} >0$ (independent of $v$) such that
\begin{multline}
\label{ineqAsymptBeforeYoung}
J_{\mu, \nu} (tu_{\varepsilon} +v)
\le\frac{c_\infty}{2}t^2-\frac{N}{2(N+\alpha)}t^{\frac{2(N+\alpha)}{N}}- \Cr{cstGap} t^2 \varepsilon^2 -\Cr{cstEquiv}\nor{v}^{\frac{2(N+\alpha)}{N}}\\
+ \Cr{cstNq} t \varepsilon^{\frac{N}{2}} \nor{v}
+ \Cl{cstremxdxr} (t^\frac{N + \alpha}{N} \nor{v}^{\frac{\alpha}{N}} + \nor{v}^\frac{N+ 2\alpha}{N}) t \varepsilon^\frac{N}{2}.
\end{multline}

By Young's inequality the following hold:
\begin{gather}
 \label{ineqNewYoung}\Cr{cstNq} t \varepsilon^{\frac{N}{2}} \nor{v}
 \le \tfrac{\Cr{cstEquiv}}{3}\nor{v}^\frac{2 (N + \alpha)}{N}  +
 \Cl{cstNewYoung} t^\frac{2(N + \alpha)}{N + 2 \alpha} \varepsilon^\frac{N (N + \alpha)}{N + 2 \alpha},\\
  \label{ineqNewYoungBisp}\Cr{cstremxdxr} \nor{v}^\frac{\alpha}{N} t^\frac{2N + \alpha}{N} \varepsilon^\frac{N}{2}
 \le \tfrac{\Cr{cstEquiv}}{3} \nor{v}^\frac{2 (N + \alpha)}{N}
 + \Cl{cstNewYoungBisp} t^\frac{2 (N + \alpha)}{N} \varepsilon^\frac{N (N + \alpha)}{2N + \alpha},\\
 \label{ineqNewYoungBis}
 \Cr{cstremxdxr} \nor{v}^{\frac{N + 2 \alpha}{N}} t \varepsilon^\frac{N}{2}
 \le \tfrac{\Cr{cstEquiv}}{3}\nor{v}^\frac{2 (N + \alpha)}{N}
 + \Cl{cstNewYoungBis} t^\frac{2 (N + \alpha)}{N} \varepsilon^{N + \alpha},
\end{gather}
from which we obtain
\begin{multline*}
J_{\mu, \nu} (tu_{\varepsilon} +v)\le(\tfrac{c_\infty}{2} - \Cr{cstGap} \varepsilon^2) t^2 - \bigl(\tfrac{N}{2(N+\alpha)} +\Cr{cstNewYoungBisp} \varepsilon^\frac{N (N + \alpha)}{2N + \alpha} + \Cr{cstNewYoungBis} \varepsilon^{N + \alpha}\bigr) t^{\frac{2(N+\alpha)}{N}} \\
+
 (\Cr{cstNewYoung} \varepsilon^\frac{N (N + \alpha)}{N + 2 \alpha} )t^\frac{2(N + \alpha)}{N + 2 \alpha}.
\end{multline*}
By uniform convergence, it follows that there exists \(t_* \in (0, 1)\) such that for sufficiently small \(\varepsilon>0\), we have
\[
 J_{\mu, \nu} (tu_{\varepsilon} +v)
 \le \frac{\alpha}{4(N+\alpha)}c_\infty^{\frac{N+\alpha}{\alpha}}.
\]
We can thus assume \(t \ge t_*\). Since \(\frac{2 (N + \alpha)}{N + 2 \alpha} \le 2\), we have
\begin{multline*}
J_{\mu, \nu} (tu_{\varepsilon} +v)
\le(\tfrac{c_\infty}{2} - \Cr{cstGap} \varepsilon^2) t^2 - \bigl(\tfrac{N}{2(N+\alpha)} + \Cr{cstNewYoungBisp} \varepsilon^\frac{N (N + \alpha)}{2N + \alpha} +  \Cr{cstNewYoungBis} \varepsilon^{N + \alpha}\bigr) t^{\frac{2(N+\alpha)}{N}} \\+
 \frac{\Cr{cstNewYoung}}{t_*^\frac{2 \alpha}{N + 2 \alpha}}  t^2 \varepsilon^\frac{N (N + \alpha)}{N + 2 \alpha}
\end{multline*}
and since \(\frac{N (N + \alpha)}{N + 2 \alpha} >  \frac{N (N + \alpha)}{2N + \alpha}\),  the conclusion follows provided 
\begin{equation*}
 \frac{N (N + \alpha)}{2N + \alpha} > 2.
\end{equation*}
Equivalently, we should have
\(\frac{1}{N + \alpha} \le \frac{1}{2} - \frac{1}{N}\),
which will be the case whenever \(N \ge 4\).
The proof is completed when \(N \ge 4\).

If \(N = 3\), the estimates above are still valid and show that there exist \(t_*, t^* \in (0, +\infty)\),
such that if \(t \in (0, t_*] \cup [t^*, +\infty)\) and \(v \in E^-\),
\[
 J_{\mu, \nu} (tu_{\varepsilon} +v)
 \le \frac{\alpha}{4(N+\alpha)}c_\infty^{\frac{N+\alpha}{\alpha}}.
\]
It remains thus to prove a strict inequality for \(t \in [t_*, t^*]\).
Since \(E^-\) is a set of negative eigenfunctions, there exists a constant \(\Cl{cstNegEig}\) such that
\[
 \int_{\RN} \abs{\nabla v}^2 + V_{\mu, \nu}
 \le -\Cr{cstNegEig} \nor{v}^2,
\]
and thus we have in place of \eqref{ineqAsymptBeforeYoung}, the following 
\begin{multline*}
J_{\mu, \nu} (tu_{\varepsilon} +v)
\le\frac{c_\infty}{2}t^2-\frac{N}{2(N+\alpha)}t^{\frac{2(N+\alpha)}{N}}- \Cr{cstGap} t^2 \varepsilon^2 -\Cr{cstEquiv}\nor{v}^{\frac{2(N+\alpha)}{N}} - \Cr{cstNegEig} \\
+ \Cr{cstNq} t \varepsilon^{\frac{N}{2}} \nor{v}
+ \Cr{cstremxdxr} (t^\frac{N + \alpha}{N} \nor{v}^{\frac{\alpha}{N}} + \nor{v}^\frac{N+ 2\alpha}{N})\, t \varepsilon^\frac{N}{2}.
\end{multline*}
We now use again estimates \eqref{ineqNewYoung} and \eqref{ineqNewYoungBis} and \eqref{ineqNewYoungBisp} replaced by the following 
\[
 \nor{v}^\frac{\alpha}{N} t^\frac{2N + \alpha}{N} \varepsilon^\frac{N}{2}
 \le \Cr{cstNegEig} \nor{v}^{2}
 + \Cl{cstNewYoungNegEig} t^\frac{2(2N + \alpha)}{2N - \alpha} \varepsilon^\frac{N^2}{2N - \alpha},
\]
to get for \(t_* \le t \le t^*\) the following estimate
\begin{multline*}
 J_{\mu, \nu} (tu_{\varepsilon} +v)\le(\tfrac{c_\infty}{2} - \Cr{cstGap} \varepsilon^2) t^2 - \bigl(\tfrac{N}{2(N+\alpha)} + \Cr{cstNewYoungNegEig} t^*{}^\frac{2 \alpha^2}{N (2N - \alpha)} \varepsilon^\frac{N^2}{2N - \alpha} +  \Cr{cstNewYoungBis} \varepsilon^{N + \alpha}\bigr) t^{\frac{2(N+\alpha)}{N}}\\
 +
 \frac{\Cr{cstNewYoung}}{t_*^\frac{2 \alpha}{N + 2 \alpha}}  t^2 \varepsilon^\frac{N (N + \alpha)}{N + 2 \alpha}.
\end{multline*}
For \(N = 3\) we have \(N^2/(2N - \alpha) > 2\) if and only if \(\alpha > \frac{3}{2}\) and this concludes the proof. 
\end{proof}

\subsection{Palais-Smale condition.} To obtain the existence of nontrivial solutions to \eqref{q1}, the following compactness result plays a crucial role. \begin{lemma}\label{ps}
$J_{\mu, \nu}$ satisfies the Palais-Smale condition in $(-\infty,c)$ if $c<\frac{\alpha}{2(N+\alpha)}c_\infty^{\frac{N+\alpha}{\alpha}}$. Namely, if $\{u_m\}_{m \in \mathbb{N}}\subset H^1(\RN)$ satisfies
$$
J_{\mu, \nu} (u_m)\to c,\,\, J_{\mu, \nu}'(u_m)\to 0\,\,\text{in $H^{-1} (\RN)$, as $m\to\infty$}
$$
then up to a subsequence, there exists $u\in H^1(\RN)$ such that $u_m\to u$ strongly in $H^1(\RN)$, as $m\to\infty$.
\end{lemma}
Before proving Lemma~\ref{ps}, let us first prove a suitable version of Lions's lemma \cite[Lemma I.1]{Lions1} in the nonlocal context.
\begin{lemma}\label{lions} Let $r>0$, $N\ge3$ and $\{u_m\}_{m \in \mathbb{N}}$ be bounded in $H^1(\RN)$, then the following are equivalent
\begin{enumerate}
 \item for some $p\in\left[\frac{N+\alpha}{N},\frac{N+\alpha}{N-2}\right)$,
 \[
\lim_{m\to \infty}\sup_{z\in\RN}\int_{B_r(z)}\int_{B_r(z)}\frac{\abs{u_m (x)}^p\abs{u_m (y)}^p}{\abs{x - y}^{N-\alpha}}\,\ud x\,\ud y=0,
 \]
 \item for some $q\in\left[2,\frac{2N}{N-2}\right)$ 
 \[
\lim_{m\to \infty}\sup_{z\in\RN}\int_{B_r(z)}|u_m|^q\,\ud x=0.
\]
\end{enumerate}
In both cases, one has for any $s\in(2,\frac{2N}{N-2})$ and $t\in(\frac{N+\alpha}{N},\frac{N+\alpha}{N-2})$
$$
\lim_{m\to \infty}\int_{\RN}|u_m|^s\,\ud x=\lim_{m\to \infty}\int_{\RN} (I_\alpha\ast|u_m|^t)|u_m|^t\,\ud x=0.
$$
\end{lemma}
\begin{proof} We only prove the necessary condition as the sufficient condition is trivial. Let $r>0, p\in[\frac{N+\alpha}{N},\frac{N+\alpha}{N-2})$ and $\{u_m\}_{m \in \mathbb{N}}$ be bounded in $H^1(\RN)$ satisfying
$$
\sup_{z\in\RN}\int_{B_r(z)}\int_{B_r(z)}\frac{\abs{u_m (x)}^p\abs{u_m (y)}^p}{\abs{x - y}^{N-\alpha}}\,\ud x\,\ud y\to 0,\,\,m\to \infty,
$$
we claim 
\begin{equation}\label{vanish}
\sup_{z\in\RN}\int_{B_r(z)}\abs{u_m (x)}^q\,\ud x\to 0,\,\,m\to \infty
\end{equation}
where $q=p\frac{2N}{N+\alpha}\in[2,\frac{2N}{N-2})$. Indeed, otherwise, up to a subsequence there exist $\delta>0$ and $\{z_m\}\subset\RN$ such that
$$
\int_{B_r(z_m)}\abs{u_m (x)}^q\,\ud x\to\delta,\,\,m\to \infty.
$$
Moreover, up to a subsequence, there exists $u\in H^1(\RN)\setminus\{0\}$ such that $u_m(\cdot+z_m)\to u$ weakly in $H^1(\RN)$ and almost everywhere  in $B_r(0)$ for any $m$. By the Hardy--Littlewood--Sobolev inequality,
$$
\int_{B_r(0)}\int_{B_r(0)}\frac{|u(x)|^p|u(y)|^p}{\abs{x - y}^{N-\alpha}}\,\ud x\,\ud y\in(0,\infty).
$$
Then by Fatou's lemma,
\begin{equation*}
 \begin{split}
\lim_{m\to \infty}\sup_{z\in\RN}\int_{B_r(z)}\int_{B_r(z)}&\frac{\abs{u_m (x)}^p\abs{u_m (y)}^p}{\abs{x - y}^{N-\alpha}}\,\ud x\,\ud y\\
&\ge\liminf_{m\to \infty}\int_{B_r(z_m)}\int_{B_r(z_m)}\frac{\abs{u_m (x)}^p\abs{u_m (y)}^p}{\abs{x - y}^{N-\alpha}}\,\ud x\,\ud y\\
&=\liminf_{m\to \infty}\int_{B_r(0)}\int_{B_r(0)}\frac{|u_m(x+z_m)|^p|u_m(y+x_m)|^p}{\abs{x - y}^{N-\alpha}}\,\ud x\,\ud y\\
&\ge\int_{B_r(0)}\int_{B_r(0)}\frac{|u(x)|^p|u(y)|^p}{\abs{x - y}^{N-\alpha}}\,\ud x\,\ud y>0,  
 \end{split}
\end{equation*}
which is a contradiction and \eqref{vanish} holds.

Finally, by Lions's lemma \cite[Lemma I.1]{Lions1} and the Hardy--Littlewood--Sobolev inequality, we have for any $s\in(2,\frac{2N}{N-2})$ and $t\in(\frac{N+\alpha}{N},\frac{N+\alpha}{N-2})$,
\begin{equation*}
\lim_{m\to \infty}\int_{\RN}|u_m|^s\,\ud x=\lim_{m\to \infty}\int_{\RN} (I_\alpha\ast|u_m|^t)|u_m|^t\,\ud x=0.\qedhere
\end{equation*}
\end{proof}
\begin{proof}[Proof of Lemma~\ref{ps}] Let $\{u_m\}_{m \in \mathbb{N}}\subset H^1(\RN)$ be a (P-S)$_c$ sequence, namely
$$
J_{\mu, \nu} (u_m)\to c<\frac{\alpha}{2(N+\alpha)}c_\infty^{\frac{N+\alpha}{\alpha}},\,\, J_{\mu, \nu}'(u_m)\to 0\,\,\text{in $H^{-1} (\RN)$, as $m\to\infty$}.
$$
Then
\begin{equation*}
 \begin{split}
  c+o_m(1)\|u_m\|&=J_{\mu, \nu} (u_m)-\frac{1}{2}\lan J_{\mu, \nu}'(u_m),u_m\ran\\
&=\frac{\alpha}{2(N+\alpha)}\int_{\RN} (I_\alpha\ast|u_m|^{\frac{N+\alpha}{N}})|u_m|^{\frac{N+\alpha}{N}},
 \end{split}
\end{equation*}
where $o_m(1)\to 0$, as $m\to\infty$. We first prove that the sequence $\{u_m\}_{m \in \mathbb{N}}$ is bounded in $H^1(\RN)$. Indeed, if not $\|u_m\|\to\infty$, as $m\to\infty$. Set $\Tilde{u}_m=u_m/\|u_m\|$, to have 
$$
\lim_{m\to \infty}\int_{\RN} (I_\alpha\ast|\Tilde{u}_m|^{\frac{N+\alpha}{N}})|\Tilde{u}_m|^{\frac{N+\alpha}{N}}=0.
$$
It follows from Lemma~\ref{lions} that $\Tilde{u}_m\to 0$ strongly in $L^s(\RN)$ for any $s\in(2,\frac{2N}{N-2})$. Then we have
$$
\int_{\RN}\frac{1}{\nu^2 + \abs{x}^2}|u_m|^2\,\ud x=\|u_m\|^2\int_{\RN}\frac{1}{\nu^2 + \abs{x}^2}|\Tilde{u}_m|^2\,\ud x=o_m(1)\|u_m\|^2,
$$
which implies 
\begin{equation*}
 \begin{split}
c+o_m(1)\|u_m\|&=J_{\mu, \nu} (u_m)-\frac{N}{2(N+\alpha)}\lan J_{\mu, \nu}'(u_m),u_m\ran\\
&=\frac{\alpha}{2(N+\alpha)}\int_{\RN}|\nabla u_m|^2+V_{\mu, \nu}|u_m|^2\\
&=\Bigl[\frac{\alpha}{2(N+\alpha)} +o_m(1)\Bigr]\int_{\RN}|\nabla u_m|^2+|u_m|^2.
 \end{split}
\end{equation*}
and the sequence $\{u_n\}_{n \in \mathbb{N}}$ stays bounded. 

Next we may assume there exists $u\in H^1(\RN)$ such that $u_m\rightharpoonup u$ weakly in $H^1(\RN)$ and almost everywhere  in $\RN$, as $m\to\infty$. Let $v_m=u_m-u$, then by Brezis-Lieb's lemma
\begin{equation*}
\int_{\RN}|\nabla u_m|^2+V_{\mu, \nu}|u_m|^2=\int_{\RN}\abs{\nabla u}^2+V_{\mu, \nu}\abs{u}^2+\int_{\RN}|\nabla v_m|^2+|v_m|^2+o_m(1)
\end{equation*}
and by \cite[Lemma 2.4]{MV3},
\begin{multline*}
\int_{\RN} (I_\alpha\ast|u_m|^{\frac{N+\alpha}{N}})|u_m|^{\frac{N+\alpha}{N}}\\
=\int_{\RN} (I_\alpha\ast\abs{u}^{\frac{N+\alpha}{N}})\abs{u}^{\frac{N+\alpha}{N}}
+\int_{\RN} (I_\alpha\ast|v_m|^{\frac{N+\alpha}{N}})|v_m|^{\frac{N+\alpha}{N}} +o_m(1).
\end{multline*}
Then
\begin{equation}\label{decomposition}
\left\{
\begin{aligned}
&c+o_m(1)=J_{\mu, \nu} (u)+\frac{1}{2}\int_{\RN}|\nabla v_m|^2+|v_m|^2-\frac{N}{2(N+\alpha)}\int_{\RN} (I_\alpha\ast |v_m|^{\frac{N+\alpha}{N}})|v_m|^{\frac{N+\alpha}{N}},\\
&o_m(1)=\lan J_{\mu, \nu}'(u),u\ran+\int_{\RN}|\nabla v_m|^2+|v_m|^2-\int_{\RN} (I_\alpha\ast |v_m|^{\frac{N+\alpha}{N}})|v_m|^{\frac{N+\alpha}{N}}.
\end{aligned}
\right.
\end{equation}
We have $J_{\mu, \nu}'(u)=0$ in $H^{-1} (\RN)$ and $J_{\mu, \nu} (u)\ge 0$. Suppose $\|v_m\|^2\to l\ge 0$, as $m\to\infty$, then by \eqref{decomposition}
$$
\lim_{m\to \infty}\int_{\RN} (I_\alpha\ast |v_m|^{\frac{N+\alpha}{N}})|v_m|^{\frac{N+\alpha}{N}}=l.
$$
If $l>0$, then by the definition of $c_\infty$, we have
$$
l+o_m(1)=\|v_m\|^2\ge\|v_m\|_2^2\ge c_\infty\left(l+o_m(1)\right)^{\frac{N}{N+\alpha}},
$$
which implies $l\ge c_\infty^{\frac{N+\alpha}{\alpha}}$. Then by \eqref{decomposition},
$$
c\ge\frac{\alpha}{2(N+\alpha)}c_\infty^{\frac{N+\alpha}{\alpha}},
$$
which is a contradiction. Therefore $l=0$ and the proof is complete.
\end{proof}

\begin{proof}[Proof of Theorem~\ref{Theorem 2}] Now, we are in position to prove Theorem~\ref{Theorem 2}. For this purpose, we recall the following critical point theorem due to P.~Bartolo, V.~Benci and D.~Fortunato \cite{Benci}.
\begin{lemma}\label{critical}{\rm \cite[Theorem 2.4]{Benci}}
Let $H$ be a real Hilbert space and $f\in C^1(H,\R)$ be a functional satisfying the following assumptions:
\begin{itemize}

\item [$({f_1})$] $f(-u)=f(u)$ for any $u\in H$ and $f(0)=0$;

\item [$({f_2})$] there exists $\beta>0$ such that $f$ satisfies the Palais-Smale condition in $(0,\beta)$;

\item [$({f_3})$] there exists two closed subspaces $V, W\subset H$ and positive constants $\rho,\delta$ such that

     \begin{itemize}

     \item [$({i})$] $f(u)<\beta$ for any $u\in W$;

     \item [$({ii})$] $f(u)\ge\delta$ for any $u\in V$ with $\nor{u}=\rho$;

     \item [$({iii})$] ${\rm codim} (V)<+\infty$.

     \end{itemize}

\end{itemize}
Then $f$ admits at least $m$ pairs of critical points with critical values belonging to the interval $[\delta,\beta)$ and
$$
m={\rm dim} (V\cap W)-{\rm codim} (V+W).
$$
\end{lemma}
Let us divide the proof of Theorem \ref{Theorem 2} into two steps: 
\medbreak
\textbf{Step 1.} We use Lemma~\ref{critical} to show that \eqref{q1} admits at least one nontrivial solution for $\mu\in[\la_n,\la_{n+1})$. Obviously, $J_{\mu, \nu} (-u)=J_{\mu, \nu} (u)$ for any $u\in H^1(\RN)$ and $J_{\mu, \nu} (0)=0$. By Lemma~\ref{ps}, $J_{\mu, \nu}$ satisfies the Palais-Smale condition in $(0,\beta)$ with $\beta=\frac{\alpha}{2(N+\alpha)}c_\infty^{\frac{N+\alpha}{\alpha}}$. Take $V=E^+$ and
$$
W=\{w\in H^1(\RN): w=tu_{\varepsilon} +v,\,\,t\in\R,\,\,v\in E^-\},
$$
then $V+W=H^1(\RN),\,\,{\rm codim} (V)=n<+\infty.$ By Lemma~\ref{lemma4}, for $\e$ small enough, $\int_{\RN}\abs{\nabla u_{\varepsilon}} +V_{\mu, \nu}\abs{u_{\varepsilon}}^2>0$, which implies $u_{\varepsilon}\not\in E^-$. Then ${\rm dim} (V\cap W)=1,\,\,m=1$. Noting that $J_{\mu, \nu}$ is even in $H^1(\RN)$, by Lemma~\ref{estimate}, for $\varepsilon > 0$ small, we have
$$
\sup_{w\in W}J_{\mu, \nu} (w)<\frac{\alpha}{2(N+\alpha)}c_\infty^{\frac{N+\alpha}{\alpha}}.
$$
On the other hand, observe that for any $u\in E^+$,
$$
\int_{\RN}\abs{\nabla u}^2+ \abs{u}^2\ge\la_{n+1}\int_{\RN}\frac{1}{\nu^2 + \abs{x}^2}\abs{u}^2\,\ud x.
$$
By the Hardy--Littlewood--Sobolev inequality, for any $u\in E^+$, we get
\begin{equation*}
\begin{split}
J_{\mu, \nu} (u)&\ge\frac{1}{2}\int_{\RN}\abs{\nabla u}^2+V_{\mu, \nu}\abs{u}^2-\mathcal{C}_\alpha\nor{u}_2^{\frac{2(N+\alpha)}{N}}\\
&\ge\frac{1}{2}\left(1-\frac{\mu}{\la_{n+1}}\right)\int_{\RN}\abs{\nabla u}^2+ \abs{u}^2-\mathcal{C}_\alpha\nor{u}_2^{\frac{2(N+\alpha)}{N}}\\
&=\nor{u}^2\left[\frac{1}{2}\left(1-\frac{\mu}{\la_{n+1}}\right)-\mathcal{C}_\alpha\nor{u}_2^{\frac{2\alpha}{N}}\right]\\
&\ge\frac{1}{4}\left(1-\frac{\mu}{\la_{n+1}}\right)\rho^2,\,\,\text{for $\nor{u}=\rho$ sufficiently small.}
\end{split}
\end{equation*}
As a consequence of Lemma~\ref{critical}, \eqref{q1} admits at least one nontrivial solution $u\in H^1(\RN)$ with $J_{\mu, \nu} (u)<\frac{\alpha}{2(N+\alpha)}c_\infty^{\frac{N+\alpha}{\alpha}}$.
\medbreak
\textbf{Step 2.} In the following, we prove the existence of ground state solutions to \eqref{q1}. Let
$$
K:=\{u\in H^1(\RN)\setminus\{0\}: J_{\mu, \nu}'(u)=0\,\,\text{in $H^{-1} (\RN)$}\},
$$
then by \textbf{Step 1}, $K\not=\emptyset$ and
$$
m_{\mu, \nu}:=\inf_{u\in K}J_{\mu, \nu} (u)<\frac{\alpha}{2(N+\alpha)}c_\infty^{\frac{N+\alpha}{\alpha}}.
$$
Obviously, $m_{\mu, \nu}\ge 0$ and there exists a sequence $\{u_m\}_{m \in \mathbb{N}}\) in $K$ such that $J_{\mu, \nu}'(u_m)=0$ in $H^{-1} (\RN)$ and $J_{\mu, \nu} (u_m)\to m_{\mu, \nu}$ as $m\to\infty$. By Lemma~\ref{ps}, up to a subsequence, there exists $u_0\in H^1(\RN)$ such that $u_m\to u_0$ strongly in $H^1(\RN)$ as $m\to\infty$. Then $u_0\in K\cup\{0\}$ and $J_{\mu, \nu} (u_0)=m_{\mu, \nu}$.

To conclude the proof, it remains to show that $m_{\mu, \nu}>0$, indeed if not, then
$$
\lim_{m\to \infty}\int_{\RN} (I_\alpha\ast |u_m|^{\frac{N+\alpha}{N}})|u_m|^{\frac{N+\alpha}{N}}=0.
$$
Similarly as above, the sequence $\{u_m\}_{m \in \mathbb{N}}$ is bounded in $H^1(\RN)$ and by virtue of Lemma~\ref{lions},
$$
\lim_{m\to \infty}\int_{\RN}\frac{1}{\nu^2 + \abs{x}^2}|u_m|^2\,\ud x=0
$$
and then $\|u_m\|\to 0$, as $m\to\infty$. Observe that $u_m=\bar{u}_m+v_m$ and $\|u_m\|^2=\|\bar{u}_m\|^2+\|v_m\|^2$, where $\bar{u}_m\in E^-$ and $v_m\in E^+$. Then $\|\bar{u}_m\|\to 0$ and $\|v_m\|\to 0$, as $m\to\infty$. If $v_m=0$, then by $u_m\not=0$, we have $\bar{u}_m\not=0$ and
\begin{equation*}
\begin{split}
J_{\mu, \nu} (u_m)&=J_{\mu, \nu} (\bar{u}_m)=\frac{1}{2}\int_{\RN}|\nabla \bar{u}_m|^2+V_{\mu, \nu}|\bar{u}_m|^2-\int_{\RN} (I_\alpha\ast |\bar{u}_m|^{\frac{N+\alpha}{N}})|\bar{u}_m|^{\frac{N+\alpha}{N}}\\
&\le\frac{1}{2}\left(1-\frac{\mu}{\la_n}\right)\int_{\RN}|\nabla \bar{u}_m|^2+|\bar{u}_m|^2-\int_{\RN} (I_\alpha\ast |\bar{u}_m|^{\frac{N+\alpha}{N}})|\bar{u}_m|^{\frac{N+\alpha}{N}}\\
&<0,
\end{split}
\end{equation*}
which contradicts the fact that
\begin{equation*}
 \begin{split}
  J_{\mu, \nu} (u_m)&=J_{\mu, \nu} (u_m)-\frac{1}{2}\lan J_{\mu, \nu}'(u_m),u_m\ran\\
&=\frac{\alpha}{2(N+\alpha)}\int_{\RN} (I_\alpha\ast|u_m|^{\frac{N+\alpha}{N}})|u_m|^{\frac{N+\alpha}{N}}>0.
 \end{split}
\end{equation*}
So we get $v_m\not=0$ for any $m$.
\medbreak
\textbf{Case 1.} Assume that up to a subsequence, $\lim\limits_{m\to \infty}\frac{\|\bar{u}_m\|}{\|v_m\|}<+\infty$, then $\|\bar{u}_m\|\le C\|v_m\|$ for any $m$. By $J_{\mu, \nu}'(u_m)=0$ in $H^{-1} (\RN)$ and $v_m\in E^+$, we have
\begin{equation}\label{bound}
\begin{split}
\int_{\RN} (I_\alpha\ast|u_m|^{\frac{N+\alpha}{N}})|u_m|^{\frac{\alpha-N}{N}}u_mv_m&=\int_{\RN}|\nabla v_m|^2+V_{\mu, \nu}|v_m|^2\\
&\ge\left(1-\frac{\mu}{\la_{n+1}}\right)\int_{\RN}|\nabla v_m|^2+|v_m|^2.
\end{split}
\end{equation}
By the Hardy--Littlewood--Sobolev inequality and H\"older's inequality,
\begin{equation*}
 \begin{split}
  \int_{\RN} (I_\alpha\ast|u_m|^{\frac{N+\alpha}{N}})|u_m|^{\frac{\alpha-N}{N}}u_mv_m
&\le\mathcal{C}_\alpha\left(\int_{\RN}|u_m|^2\right)^{\frac{N+\alpha}{2N}}\left(\int_{\RN}|u_m|^{\frac{2\alpha}{N+\alpha}}|v_m|^{\frac{2N}{N+\alpha}}\right)^{\frac{N+\alpha}{2N}}\\
&\le\mathcal{C}_\alpha\|u_m\|_2^{\frac{N+2\alpha}{N}}\|v_m\|_2\le c\|v_m\|^{\frac{2(N+\alpha)}{N}},
 \end{split}
\end{equation*}
where $c>0$ (independent of $m$). By \eqref{bound},
$$
\left(1-\frac{\mu}{\la_{n+1}}\right)\|v_m\|^2\le c\|v_m\|^{\frac{2(N+\alpha)}{N}},
$$
which is a contradiction, since $\mu<\la_{n+1}$ and $\|v_m\|\to 0$, as $m\to\infty$. Thus, $m_{\mu, \nu}>0$.
\medbreak
\textbf{Case 2.} Assume that, up to a subsequence, $\lim\limits_{m\to \infty}\frac{\|\bar{u}_m\|}{\|v_m\|}=\infty$, then $\bar{u}_m\not=0$ for $m$ large and $\lim\limits_{m\to \infty}\frac{\|v_m\|}{\|\bar{u}_m\|}=0$. By $J_{\mu, \nu}'(u_m)=0$ in $H^{-1} (\RN)$, we have
\begin{equation}
\label{bound1}
\begin{split}
 \int_{\RN}|\nabla \bar{u}_m|^2&+V_{\mu, \nu}|\bar{u}_m|^2\\
 &=\int_{\RN} (I_\alpha\ast|u_m|^{\frac{N+\alpha}{N}})|u_m|^{\frac{\alpha-N}{N}}u_m\bar{u}_m\\
 &=\int_{\RN} (I_\alpha\ast|u_m|^{\frac{N+\alpha}{N}})|u_m|^{\frac{N+\alpha}{N}}-
\int_{\RN} (I_\alpha\ast|u_m|^{\frac{N+\alpha}{N}})|u_m|^{\frac{\alpha-N}{N}}u_mv_m.
\end{split}
\end{equation}
By Lemma~\ref{norm}, we have $\left|\|u_m\|_\ast-\|\bar{u}_m\|_\ast\right|\le\|v_m\|_\ast$
and then by $\|v_m\|=o(\|\bar{u}_m\|)$,
$$
\int_{\RN} (I_\alpha\ast|u_m|^{\frac{N+\alpha}{N}})|u_m|^{\frac{N+\alpha}{N}}=\|\bar{u}_m\|_\ast^{\frac{2(N+\alpha)}{N}} (1+o_m(1)).
$$
At the same time, similarly as above, for some $c>0$ we have 
$$
\int_{\RN} (I_\alpha\ast|u_m|^{\frac{N+\alpha}{N}})|u_m|^{\frac{\alpha}{N}}|v_m|\le c\|u_m\|^{\frac{N+2\alpha}{N}}\|v_m\|=o(\|\bar{u}_m\|_\ast^{\frac{2(N+\alpha)}{N}}),
$$
where we used the fact that norms in $E^-$ are equivalent. Then by \eqref{bound1}, for $m$ large enough,
$$
\int_{\RN}|\nabla \bar{u}_m|^2+V_{\mu, \nu}|\bar{u}_m|^2=\|\bar{u}_m\|_\ast^{\frac{2(N+\alpha)}{N}} (1+o_m(1))>0,
$$
which contradicts the fact that
$$
\int_{\RN}|\nabla \bar{u}_m|^2+V_{\mu, \nu}|\bar{u}_m|^2\le0,\,\,\text{since $\bar{u}_m\in E^-$ and $\mu\ge\la_n$}.
$$
Thus $m_{\mu, \nu}>0$ and the proof of Theorem \ref{Theorem 2} is now complete.
\end{proof}


\section{Proof of Theorem~\ref{Theorem 3}.}

Finally we establish an upper bound for the value $\mu^{\nu}$.

\begin{proof}[Proof of Theorem~\ref{Theorem 3}]
For any $p>\max\{2,N/4\}$, let
$$
u_p(x)=\frac{\nu^{2p}}{(\nu^2 + \abs{x}^2)^p},\,\,x\in\RN,
$$
then
$$
|\nabla u_p(x)|=\frac{2p\nu^{2 p}\abs{x}}{(\nu^2 + \abs{x}^2)^{p+1}},\,\,x\in\RN,
$$
and by the change of variables \(r = \nu \sqrt s\), we get 
\begin{equation*}
\begin{split}
 \int_{\RN}|\nabla u_p|^2+|u_p|^2&=C_N\left[4p^2 \int_0^\infty\frac{\nu^{4 p} r^{N+1}}{(\nu^2 + r^2)^{2p+2}}\,\ud r+ \int_0^\infty\frac{\nu^{4 p} r^{N-1}}{(\nu^2 + r^2)^{2p}}\,\ud r\right]\\
&=\frac{1}{2}C_N \left[4p^2\nu^{N-2}\int_0^\infty\frac{s^{N/2}}{(1+s)^{2p+2}}\,\ud s+\nu^N\int_0^\infty\frac{s^{N/2-1}}{(1+s)^{2p}}\,\ud s\right]
\end{split}
\end{equation*}
and 
\begin{equation*}
\begin{split}
\int_{\RN}\frac{|u_p|^2}{\nu^2 + \abs{x}^2}\,\ud x&=C_N \int_0^\infty\frac{\nu^{4p}r^{N-1}}{(\nu^2 + r^2)^{2p+1}}\,\ud r=\frac{1}{2}C_N\nu^{N - 2} \int_0^\infty\frac{s^{N/2-1}}{(1+s)^{2p+1}}\,\ud s.
\end{split}
\end{equation*}
By the definition of the Beta function, for any $x,y>0$,
$$
B(x,y)=\int_0^\infty\frac{t^{x-1}}{(1+t)^{x+y}}\,\ud t,\,\,B(x,y)=\frac{\Gamma(x)\Gamma(y)}{\Gamma(x+y)}.
$$
Recalling that $4p>N$,
\begin{equation*}
\begin{split}
\frac{\int_{\RN}|\nabla u_p|^2+|u_p|^2}{\int_{\RN}\frac{|u_p|^2}{\nu^2 + \abs{x}^2}\,\ud x}&=\frac{4 p^2 B(\frac{N}{2} +1,2p+1-\frac{N}{2})+\nu^2 B(\frac{N}{2},2p-\frac{N}{2})}{B(\frac{N}{2},2p+1-\frac{N}{2})}\\
&=\frac{2Np^2}{2p+1} +\frac{4\nu^2 p}{4p-N}.
\end{split}
\end{equation*}
It follows that
$$
\mu^{\nu}\le\min_{p>N/4}\left(\frac{2Np^2}{2p+1} +\frac{4p\nu^2}{4p-N}\right).
$$
In particular, if we take \(p = \frac{N}{4} + 1\), we obtain
\[
 \mu^{\nu}\le \frac{N (N + 4)^2}{4(N + 6)} + \frac{(N + 4)\nu^2}{4},
\]
which together with \eqref{lower_est_asym} yields 
\[
\lim_{N\to \infty}\frac{\mu^{\nu}}{\frac{N^2(N-2)}{4(N+1)}}=1.
\]

\end{proof}

\end{document}